\date{May 2, 2012}
\documentclass[11pt]{amsart}
\usepackage{latexsym,amsmath,amsfonts,amscd,amssymb}
\usepackage{graphics}
\usepackage[all]{xy}
\newtheorem{theorem}{Theorem}[section]
\newtheorem{lemma}[theorem]{Lemma}
\newtheorem{proposition}[theorem]{Proposition}

\newtheorem{conjecture}[theorem]{Conjecture}

\newtheorem{example}[theorem]{Example}

\theoremstyle{remark}
\newtheorem{remark}[theorem]{Remark}

\newcommand{\la}{\langle}
\newcommand{\ra}{\rangle}
\newcommand{\too}{\longrightarrow}
\newcommand{\surj}{\twoheadrightarrow}
\newcommand{\inc}{\hookrightarrow}

\newcommand{\ox}{\otimes}

\newcommand{\quism}{\stackrel{\sim}{\too}}

\newcommand{\Tor}{\operatorname{Tor}}

\newcommand{\length}{\operatorname{length}}

\newcommand{\cF}{{\mathcal F}}

\newcommand{\cM}{{\mathcal M}}

\newcommand{\QQ}{{\mathbb Q}}

\newcommand{\bk}{{\mathbf{k}}}

\begin{document}

\title[The Hilali Conjeture]{The Hilali Conjecture for hyperelliptic spaces}



\author[J. Fern\'{a}ndez-Bobadilla]{Javier Fern\'{a}ndez de Bobadilla}
\address{Instituto de Ciencias Matem\'aticas
CSIC-UAM-UC3M-UCM, Consejo Superior de Investigaciones Cient\'{\i}ficas,
C/ Nicol\'as Cabrera, n$^o$ 13-15, Campus Cantoblanco UAM, 28049 Madrid, Spain}

\email{javier@icmat.es}

\author[J. Fres\'{a}n]{Javier Fres\'{a}n}
\address{LAGA, UMR 7539, Institut Galil\'{e}e, Universit\'{e} Paris 13, 99, Avenue Jean-Baptiste Cl\'{e}ment, F-93430,
Villetaneuse, France}

\email{fresan@math.univ-paris13.fr}

\author[V. Mu\~{n}oz]{Vicente Mu\~{n}oz}
\address{Facultad de Ciencias
Matem\'aticas, Universidad Complutense de Madrid, Plaza de Ciencias
3, 28040 Madrid, Spain}

\email{vicente.munoz@mat.ucm.es}

\author[A. Murillo]{Aniceto Murillo}
\address{Departamento de \'Algebra, Geometr\'{\i}a y Topolog\'{\i}a, Universidad de M\'alaga, Ap. 59, 29080 M\'alaga, Spain}
\email{aniceto@uma.es}

\thanks{Fourth author partially supported through Spanish MEC grant MTM2010-17389.}

\maketitle

\begin{abstract}
Hilali Conjecture predicts that for a simply-connected elliptic space,
the total dimension
of the rational homotopy does not exceed that of the rational homology.
Here we give a proof of this conjecture for a class of elliptic spaces
known as hyperelliptic.
\end{abstract}

%%%%%%%%%%%%%%%%%%%%%%%%%%%%%%%%%%%%%%%%%%%%%%%%%%%%%%%%%%%%%%%%%%
\section{Introduction} \label{sec:introduction}
%%%%%%%%%%%%%%%%%%%%%%%%%%%%%%%%%%%%%%%%%%%%%%%%%%%%%%%%%%%%%%%%%%

Let $X$ be a simply-connected CW-complex. Then $X$ is said to be of elliptic
type if both $\dim H^*(X,\QQ)<\infty$ and $\dim \pi_*(X)\ox \QQ <\infty$.
For these spaces, Hilali conjetured in \cite{thesis} the following:

\begin{conjecture}\label{conj:Hilali}
  If $X$ is a simply-connected CW-complex of elliptic type, then
   $$
   \dim \pi_*(X)\ox \QQ \leq \dim H^*(X,\QQ) \, .
   $$
\end{conjecture}

By the theory of minimal models
of Sullivan \cite{GM}, the rational homotopy type of $X$ is encoded in a differential
algebra $(A,d)$ called {\em the minimal model} of $X$. This is a free graded algebra
$A=\Lambda V$, generated by a graded vector space $V=\bigoplus_{k\geq 2} V^k$, and
with decomposable differential, i.e., $d:V^k \to (\Lambda^{\geq 2} V)^{k+1}$. It
satisfies that:
  \begin{eqnarray*}
   V^k &=& (\pi_k(X)\ox \QQ)^* \, ,\\
   H^k(\Lambda V, d) & = & H^k(X,\QQ) \, .
  \end{eqnarray*}

Therefore the Hilali conjecture can be rewritten as follows: for a
finite-dimensional graded vector space
$V$ (in degrees bigger or equal than two), we have
 $$
 \dim V \leq \dim H^*(\Lambda V,d)
 $$
for any decomposable differential $d$ on $\Lambda V$.

\medskip

An elliptic space $X$ is called of pure type if its minimal model $(\Lambda V,d)$ satisfies that
$V=V^{even}\oplus V^{odd}$, $d(V^{even})=0$ and $d(V^{odd}) \subset \Lambda V^{even}$. Also
$X$ is called hyperelliptic if $d(V^{even})=0$ and $d(V^{odd}) \subset \Lambda^+ V^{even}\otimes
\Lambda V^{odd}$.

In his thesis \cite{thesis} in 1990, Hilali proved Conjecture \ref{conj:Hilali} for elliptic spaces
of pure type. The conjecture is known to hold \cite{Hi2,Hi1} also in several cases: H-spaces,
nilmanifolds, symplectic and cosymplectic manifolds, coformal spaces with only odd-degree generators,
and formal spaces. Hilali and Mamouni \cite{Hi2,Hi1} have also proved Conjecture \ref{conj:Hilali}
for hyperelliptic spaces under various conditions in the homotopical and homological Euler characteristics.

The main result of this paper is the following:

\begin{theorem} \label{thm:main}
 Conjecture \ref{conj:Hilali} holds for hyperelliptic spaces.
\end{theorem}

We shall start by proving it for elliptic spaces of pure type in
section \ref{sec:proof}. This requires to reduce the question to a problem
about Tor functors of certain modules of finite length over a polynomial
ring. We solve it by using a semicontinuity result for the Tor functor.
Then in section \ref{sec:proof2} we prove theorem \ref{thm:main}
for hyperelliptic spaces. For this we have to prove a semicontinuity result
for the homology of elliptic spaces, and apply it to reduce the general
case to the case in which the minimal model only has generators of odd degree
and zero differential. We give two different proofs of an inequality from
which the result follows.

%%%%%%%%%%%%%%%%%%%%%%%%%%%%%%%%%%%%%%%%%%%%%%%%%%%%%%%%%%%%%%%%%%%%
\section{Minimal models}\label{sec:minimal-models}
%%%%%%%%%%%%%%%%%%%%%%%%%%%%%%%%%%%%%%%%%%%%%%%%%%%%%%%%%%%%%%%%%%%%

We recall some definitions and results about minimal models
\cite{FHT}. Let $(A,d )$ be a {\it differential algebra}, that
is, $A$ is a (positively) graded commutative algebra over the rational numbers,
with a differential $d $ which is a derivation, i.e., $d (a\cdot b)
= (d  a)\cdot b +(-1)^{\deg (a)} a\cdot (d  b)$, where $\deg(a)$
is the degree of $a$. We say that
$A$ is connected if $A^0= \QQ$, and simply-connected if moreover $A^1=0$.

A simply-connected differential algebra $(A,d )$ is said to be {\it minimal\/} if:
\begin{enumerate}
 \item $A$ is free as an algebra, that is, $A$ is the free
 algebra $\Lambda V$ over a graded vector space $V=\oplus_{k\geq2} V^k$, and
 \item For $x\in V^k$, $dx \in (\Lambda V)^{k+1}$ has no linear term, i.e., it
 lives in ${\Lambda V}^{>0} \cdot {\Lambda V}^{>0} \subset {\Lambda V}$.
\end{enumerate}

Let $(A,d)$ be a simply-connected differential algebra. A minimal model for $(A,d)$ is
a minimal algebra $(\Lambda V,d)$ together with a quasi-isomorphism
$\rho: (\Lambda V,d)\to (A,d)$ (that is, a map of differential algebras such that
$\rho_*: H^*(\Lambda V,d)\to H^*(A,d)$ is an isomorphism). A minimal model for $(A,d)$
exists and it is unique up to isomorphism.

Now consider a simply-connected CW-complex $X$. There is an algebra of piecewise polynomial
rational differential forms $(\Omega^*_{PL}(X),d)$ defined in \cite[Chap. VIII]{GM}. A minimal model of $X$ is a minimal model $(\Lambda V_X,d)$ for $(\Omega^*_{PL}(X),d)$. We have that
  \begin{eqnarray*}
   V^k &=& (\pi_k(X)\ox \QQ)^* \, ,\\
   H^k(\Lambda V, d) & = & H^k(X,\QQ) \, .
  \end{eqnarray*}

A space $X$ is elliptic \cite{Felix} if both $\sum \dim \pi_k(X)\ox \QQ<\infty$ and $\sum \dim H^k (X, \QQ)<\infty$. Equivalently, if
$(\Lambda V,d)$ is the minimal model, we require that both
$V$ and $H^*(\Lambda V,d)$ are finite dimensional. For elliptic spaces, the Euler-Poincar\'{e} and the homotopic characteristics are well defined:
\begin{align*}
\chi=&\sum_{i\geq 0} (-1)^i \dim H^i(\Lambda V,\QQ), \\ \chi_\pi=&\sum_{i\geq 0} (-1)^i \dim \pi_i(X)\otimes \QQ=\dim V^{even}-\dim V^{odd}.
\end{align*}

We refer the reader to \cite[Thm. 32.10]{FHT} for the proof of the following:

\begin{proposition}\label{prop:mod}
Let $(\Lambda V, d)$ be an elliptic minimal model. Then
$\chi\geq 0$ and $\chi_\pi\leq 0.$ Moreover, $\chi_\pi<0$ if and only if $\chi=0$.
\end{proposition}

In his thesis \cite{thesis}, M. Hilali conjectured that for elliptic spaces:
   $$
   \dim \pi_*(X)\ox \QQ \leq \dim H^*(X,\QQ) \, .
   $$
In algebraic terms, this is equivalent to
 $$
 \dim V \leq \dim H^*(\Lambda V,d) \, ,
 $$
whenever $(\Lambda V,d)$ is a minimal model with $\dim V<\infty$. Note that finiteness of
both $\dim H^\ast(X,\QQ)$ and $\dim \pi_\ast(X)\otimes \QQ$ is necessary. Otherwise, one can
easily construct counterexamples such as $X=S^3 \vee S^3$.

%%%%%%%%%%%%%%%%%%%%%%%%%%%%%%%%%%%%%%%%%%%%%%%%%%%%%%%%%%%%%%%%%%%%
\section{Proof of the Hilali conjecture for elliptic spaces of pure type} \label{sec:proof}
%%%%%%%%%%%%%%%%%%%%%%%%%%%%%%%%%%%%%%%%%%%%%%%%%%%%%%%%%%%%%%%%%%%%

A minimal model $(\Lambda V,d)$ is of \emph{pure type} if $V=V^{even}\oplus V^{odd}$,
with
 $$
  d(V^{even})=0, \quad d(V^{odd}) \subset \Lambda V^{even}.
 $$
An elliptic space
is of pure type if its minimal model is so. These spaces are
widely studied in \cite[\textsection 32]{FHT}. By proposition \ref{prop:mod},
we have that $\dim V^{even}-\dim V^{odd} \leq 0$. Let $n=\dim V^{even}$ and
$n+r=\dim V^{odd}$, where $r\geq 0$.
Write $x_1,\ldots, x_n$ for
the generators of even degree, and $y_1,\ldots, y_{n+r}$ for the generators of
odd degree. Then $dx_i=0$, and $dy_j=P_j(x_1,\ldots, x_n)$, where $P_j$ are
polynomials without linear terms.

In this section we prove the following:

\begin{theorem}\label{thm:sec-proof}
 The Hilali conjecture holds for elliptic spaces of pure type.
\end{theorem}

\subsection{Expressing the homology as a Tor functor}

To work over nice modules we would like to reorder the generators $y_1,\ldots,y_{n+r}$,
so that $P_1,\ldots, P_n$ form a regular sequence in $\Lambda (x_1,\ldots,x_n).$ Recall that this means that the image of $P_i$ in $\Lambda (x_1,\ldots,x_n)/(P_1,\ldots,
P_{i-1})$ is not a zero divisor, for any $i=1,\ldots, n$. But this is
not possible in general, as shown by the following example.

\begin{example} Let $V=\QQ\langle x_1,x_2,y_1, y_2, y_3 \rangle$, where $\deg(x_1)=2$ and
$\deg(x_2)=6.$ Define a differential $d$ on $\Lambda V$ by
$$
dy_1=x_1^6+x_2^2, \quad dy_2=x_1^9+x_2^3, \quad dy_3=x_1^4x_2+x_1x_2^2.
$$
Then $(\Lambda V, d)$ is a pure minimal model. It can be proved that is elliptic if
and only if there exist exact powers of $x_1$ and $x_2$. This is the case, since
$2x_1^{10}=d(x_1^4y_1+x_1y_2-x_2y_3)$ and $2x_2^4=d(x_2^2y_1+x_2y_2-x_1^5y_3)$.
But for the same reason, models $(\Lambda (x_1,x_2,y_i,y_j), d)$ are not elliptic
for any choice of indices $i, j$. This amounts to say that $dy_i, dy_j$ are not a regular sequence in $\Lambda (x_1,x_2).$
\end{example}

However, Halperin showed in \cite[Lemma 8]{Hal} that pure models always admit a basis $z_1,\ldots,z_{n+r}$ of $V^{odd}$
such that $dz_1,\ldots, dz_{n}$ is a regular sequence in $\Lambda (x_1,\ldots,x_n)$. This basis is not necessarily
homogeneous but it is possible to preserve the lower grading induced by the number of odd elements, that is
$$
(\Lambda V)^p_q=(\Lambda V^{even} \otimes \Lambda^q V^{odd})^p.
$$
This grading passes to cohomology and by taking into account the quasi-isomorphisms
\begin{align*}
(\Lambda (x_1,\ldots,x_n,y_1,\ldots,y_{n+r}), \ d) &\stackrel{\sim}{\longrightarrow} (\Lambda (x_1,\ldots,x_n,z_1,\ldots,z_{n+r}), \ d) \\ (\Lambda (x_1,\ldots,x_n,z_1,\ldots,z_n),\ d)&\stackrel{\sim}{\longrightarrow}(\Lambda (x_1,\ldots,x_n)/(dz_1,\ldots,dz_n), \ d)
\end{align*}
with respect to the lower grading, one deduces that:
$$
H_\ast(\Lambda V, d)\cong H_\ast(\Lambda (x_1,\ldots,x_n)\slash (dz_1,\ldots,dz_n) \otimes \Lambda (z_{n+1},\ldots,z_{n+r}), d).
$$

So let $z_1,\ldots,z_{n+r}$ be a basis such that $dz_1,\ldots,dz_{n}$ form a regular sequence.
Put $P_j=dz_j$ for $j=1,\ldots,n+r$ and consider the module
 $$
 M=\QQ[x_1,\ldots, x_n]/(P_1,\ldots, P_n)
 $$
over  the ring
  $$
R=\QQ[x_1,\ldots, x_n] \, .
 $$

Consider the ring
 $$
 S=\QQ[\lambda_1,\ldots, \lambda_r]
 $$
and the map $f:S\to R$, $\lambda_i \mapsto P_{n+i}$. Then $M$ becomes an
$S$-module.

Consider also the $S$-module
 $$
 \QQ_0=S/(\lambda_1,\ldots, \lambda_r).
 $$
Then we have the following:

\begin{proposition} \label{prop:tor}
  $H_*(\Lambda V,d) \cong \Tor^*_S(M,\QQ_0)$.
\end{proposition}

\begin{proof}
  Let $U=\la z_1,\ldots, z_n \ra$, $W=\la z_{n+1},\ldots, z_{n+r}\ra$ so
  that $V^{odd}=U\oplus W$. Then the map $(\Lambda V^{even}\oplus U,d) \to (M,0)$ is a quasi-isomorphism. Actually, the Koszul complex
  $$
  R\ox \Lambda^n U \to R\ox \Lambda^{n-1} U \to
  \ldots \to R\ox \Lambda^1 U \to R \to M
  $$
is exact, which means that $(R\ox \Lambda U,d) \quism (M,0)$.

Therefore
 \begin{equation}\label{eqn:1}
 (\Lambda V,d)=(R \ox \Lambda U \ox \Lambda W,d) \quism (M\ox \Lambda W,d')\, ,
 \end{equation}
is an isomorphism,
where the differential $d'$ is defined as zero on $M$, and $d' z_{n+i}=\bar{P}_{n+i}\in M$. This
can be seen as follows: the map (\ref{eqn:1}) is a map of differential algebras. Grading both
algebras in such a way that $\Lambda^k W$ has degree $k$, we get two spectral sequences. The
map between their $E_1$-terms is
  $$
  H^* (R \ox \Lambda U,d) \ox \Lambda W \to  M \ox \Lambda W \, .
  $$
As this is an isomorphism, it follows that the map in the $E_\infty$-terms is also an isomorphism.
The $E_\infty$-terms are the homology of both algebras in (\ref{eqn:1}). So the map
(\ref{eqn:1}) is an isomorphism.

Finally, we have to identify $H^*(M\ox \Lambda W,d')\cong \Tor^*_S(M,\QQ_0)$. Note that the
homology of $(M\ox \Lambda W,d')$ is computed as follows: take the Koszul complex
  $$
  S\ox \Lambda^r W \to S\ox \Lambda^{r-1} W \to
  \ldots \to S\ox \Lambda^1 W \to S \to \QQ_0\, ,
  $$
and tensor it with $M$ over $S$ (with the $S$-module structure given above), to get
 $$
 (M\ox_S (S \ox \Lambda W),d') =(M\ox \Lambda W, d').
 $$
The homology of this computes $\Tor^*_S(M,\QQ_0)$.
\end{proof}

\begin{lemma} \label{lem:0andr}
Under our assumptions,
 $$\dim \Tor^0_S(M,\QQ_0)\geq n+1 \quad and \quad \dim \Tor^r_S(M,\QQ_0)\geq n+1.$$
\end{lemma}

\begin{proof}
  Clearly,
   $$
   \Tor^0_S(M,\QQ_0) = M \ox_S \QQ_0=M/(\bar{P}_{n+1},\ldots, \bar{P}_{n+r})=R/(P_1,\ldots, P_{n+r})\, .
   $$
  As all the polynomials $P_1,\ldots, P_{n+r}$ have no linear part, this module contains the
  constant and linear monomials at least, so $\dim \Tor^0_S(M,\QQ_0)\geq n+1$.

  For the other inequality, note that $\Tor^r_S(M, \QQ_0)$ is the kernel of
   $M\ox \Lambda^r W \to M\ox \Lambda^{r-1} W$, i.e., the kernel of
   \begin{equation}\label{eqn:2}
   (P_{n+1},\ldots, P_{n+r}) : M \to  M \oplus \stackrel{(r)}{\ldots} \oplus M\,.
   \end{equation}
  Now we use the following fact: as $M$ is a complete intersection $R$-module (it is the quotient of
  $R$ by a regular sequence), it has Poincar\'e duality in the sense that there is a map $M\to \QQ$ such that
  $\Gamma:M\ox M \stackrel{\text{mult}}{\too} M\to \QQ$ is a perfect pairing. Take elements
  $\nu, \mu_j \in M$, $j=1,\ldots,n$, such that
   \begin{eqnarray*}
   &&\Gamma(\nu, x_j)=0, \ j=1,\ldots,n,\qquad  \qquad \Gamma(\nu,1)=1, \\
   &&\Gamma(\mu_j,x_k)=\delta_{jk}, \ j,k=1,\ldots,n,\qquad \Gamma(\mu_j,1)=0, \\
   &&\Gamma(\nu, Q)=\Gamma(\mu_j, Q)=0, \ \text{ for any quadratic $Q\in R$.}
   \end{eqnarray*}
 Since the elements $\nu, \mu_j$ are in the kernel of (\ref{eqn:2}) and they are linearly independent,
we get $\dim \Tor^r_S(M,\QQ_0)\geq n+1$.
\end{proof}

\subsection{Semicontinuity theorem}
We are going to prove a semicontinuity theorem for the Tor functors $\Tor^k_S(M,\QQ_0)$
for flat families of modules $M$ of finite length (i.e., finite-dimensional as $\QQ$-vector spaces).

Consider a variable $t$. A family of $S$-modules is a module $\cM$ over $S[t]$ such that for each
$t_0$, the $S$-module
 $$
 M_{t_0} = \cM /(t-t_0)
 $$
is of finite length. We say that $\cM$ is flat over $\QQ[t]$ if it is a flat $\QQ[t]$-module, under
the inclusion $\QQ[t]\inc S[t]$. Consider $\cM$ as a $\QQ[t]$-module. Then
  $$
  \cM \cong \QQ[t]^N \oplus \frac{\QQ[t]}{(t-t_1)^{b_1}} \oplus \ldots \oplus \frac{\QQ[t]}{(t-t_l)^{b_l}} \, ,
  $$
for some $N\geq 0$, $l\geq 0$, $1\leq b_1\leq \ldots \leq b_l$.
The module is flat if and only if there is no torsion part, i.e., $l=0$ (to see this, tensor the exact
sequence $0\to \QQ[t] \stackrel{t-t_i}{\too} \QQ[t] \to \QQ[t]/(t-t_i)\to 0$ with $\cM$). Note that for
generic $\xi$, $\length (M_{\xi})=N$. Therefore the flatness is equivalent to $\cM/(t-t_i)$ being of length
$N$, i.e.,
 $$
 \cM \text{ is flat } \Longleftrightarrow \length (M_t)=N, \, \forall t\, .
 $$

\begin{lemma} \label{lem:semincontinuity}
  For any flat family $\cM$,
  $$\dim \Tor^k_S(M_0,\QQ_0)\geq \dim \Tor^k_S(M_\xi,\QQ_0),$$
for generic $\xi \in \QQ$.
\end{lemma}

\begin{proof}
  Let us resolve $\cM$ as a $S[t]$-module:
 \begin{equation}\label{eqn:3}
 0 \to S[t]^{a_r} \to \ldots \to S[t]^{a_0} \to \cM \to 0\, .
  \end{equation}

As $\cM$ is flat as $\QQ[t]$-module, if we tensor the inclusion $\QQ[t] \stackrel{t}{\inc}
\QQ[t]$ by $\cM$ over $\QQ[t]$, we have that $\cM\stackrel{t}{\inc} \cM$ is an inclusion.
Hence the sequence
  $$
  0\to \cM \stackrel{t}{\inc} \cM \to \cM/(t) \to 0
  $$
is exact. But this sequence is the sequence $0\to S[t]\to S[t]\to S[t]/(t) \to 0$ tensored
by $\cM$ over $S[t]$. Hence $\Tor^1_{S[t]}(\cM, S[t]/(t))=0$. Obviously
$\Tor^j_{S[t]}(\cM, S[t]/(t))=0$ for $j\geq 2$ (since the resolution $S[t]/(t)$ has two
terms).

Using the above, we can tensor (\ref{eqn:3})$\otimes_{S[t]} S[t]/(t)$ to get an exact
sequence:
 \begin{equation}\label{eqn:4}
 0 \to S^{a_r} \to \ldots \to S^{a_0} \to M_{0} \to 0\, .
  \end{equation}
Now we tensor (\ref{eqn:4}) by $\otimes_S \QQ_0$ and take homology to obtain
$\Tor^*_S(M_0,\QQ_0)$. But
  $$
  (\ref{eqn:4}) \otimes_S \QQ_0 = (\ref{eqn:3}) \otimes_{S[t]} \QQ_0=
  ((\ref{eqn:3}) \otimes_{S[t]} \QQ[t]) \otimes_{\QQ[t]} \QQ[t]/(t) =
 (\ref{eqn:5}) \otimes_{\QQ[t]} \QQ[t]/(t)\, ,
  $$
where $\QQ_0=S[t]/(\lambda_1,\ldots, \lambda_r, t)$, and
 \begin{equation}\label{eqn:5}
 0 \to \QQ[t]^{a_r} \to \ldots \to \QQ[t]^{a_0} \to \cF=\cM/(\lambda_1,\ldots, \lambda_r) \to 0  .
  \end{equation}
(This is just a complex, maybe not exact.) Analogously,
  $$
 \Tor^*_S(M_0,\QQ_\xi) = H^*((\ref{eqn:5}) \otimes_{\QQ[t]} \QQ[t]/(t-\xi))\, .
 $$

\medskip

So it remains to see that for a complex $L_\bullet$ of free $\QQ[t]$-modules like (\ref{eqn:5}), it holds that
 $$
 \dim H^k(L_\bullet \ox \QQ[t]/(t-\xi)) \leq \dim H^k(L_\bullet \ox \QQ[t]/(t)),
 $$
for generic $\xi$. (Tensor products are over $\QQ[t]$, which we omit in the notation henceforth.)
For proving this, just split (\ref{eqn:5}) as short exact sequences
 \begin{equation}\label{eqn:5a}
 0\to Z_i\to L_i\to B_{i-1}\to 0,
 \end{equation}
and
note that $Z_i, B_i$ are free $\QQ[t]$-modules, being submodules of free modules.
So $Z_i=\QQ[t]^{z_i}$ and $B_i=\QQ[t]^{b_i}$. Now $0\to B_i\to Z_i\to H^i(L_\bullet)\to 0$ gives
that
 $$
 H^i(L_\bullet)=\QQ[t]^{z_i-b_i}\oplus \text{torsion}.
 $$
For generic $\xi$, we have $\dim H^i (L_\bullet \ox \QQ[t]/(t-\xi))=z_i-b_i$. Hence
  $$
  \begin{array}{ccccccccc}
 0 &\to& Z_i\ox \QQ[t]/(t) &\to &L_i\ox \QQ[t]/(t)&\to& B_{i-1}\ox \QQ[t]/(t)&\to& 0 \\
 & & \downarrow & & || & & \downarrow \\
 0 &\to& Z_i(L_\bullet\ox \QQ[t]/(t)) &\to &L_i\ox \QQ[t]/(t)&\to& B_{i-1}(L_\bullet\ox \QQ[t]/(t))&\to& 0\, .
  \end{array}
  $$
The first sequence is (\ref{eqn:5a}) tensored by $\QQ[t]/(t)$. Thus the last vertical map
is surjective, and the first vertical map is injective.

Therefore, we get:
 \begin{align*}
 \dim H^i(L_\bullet \ox \QQ[t]/(t)) &= \dim  Z_i(L_\bullet \ox \QQ[t]/(t))  - \dim B_i(L_\bullet \ox \QQ[t]/(t)) \\
 &\geq \dim  Z_i \ox \QQ[t]/(t) -\dim  B_i \ox \QQ[t]/(t) \\
 &= \dim H^i(L_\bullet)\ox \QQ[t]/(t) - \dim  \Tor_1^{\QQ[t]}(H^i(L_\bullet), \QQ[t]/(t))\\
 &= z_i-b_i\, ,
 \end{align*}
where we have used in the third line that there is an exact sequence
 $$
 0\to \Tor_1^{\QQ[t]}(H^i(L_\bullet), \QQ[t]/(t))
 \to B_i\ox \QQ[t]/(t) \to Z_i\ox \QQ[t]/(t) \to H^i(L_\bullet)\ox \QQ[t]/(t) \to 0,
 $$
and in the fourth line that $\dim (N \ox \QQ[t]/(t))= \dim  \Tor_1^{\QQ[t]}(N, \QQ[t]/(t))$
for a torsion $\QQ[t]$-module $N$.
\end{proof}

\subsection{Proof of theorem \ref{thm:sec-proof}}
We proceed to the proof of the Hilali conjecture for elliptic spaces of pure type.
We have to prove that
 $$
 \dim H^*(\Lambda V,d)\geq 2n+r.
 $$
By proposition \ref{prop:tor}, we need to prove that $\dim \Tor^*_S(M,\QQ_0)\geq 2n+r$. Consider the family
  $$
   \cM=\frac{\QQ[t,x_1,\ldots,x_n]}{(P_1+tx_1,\ldots,P_n+tx_n)}\, .
  $$
For small $t$, the hypersurfaces $P_1+tx_1,\ldots,P_n+tx_n$ intersect in $N$ points
near the origin accounted with multiplicity, where $N=\length (M)$. Therefore
$\cM$ is a flat family. By lemma \ref{lem:semincontinuity}, it is enough to bound
below $\dim \Tor^*_S(M_\xi,\QQ_0)$. But for generic $t$, the hypersurfaces
$P_1+tx_1,\ldots,P_n+tx_n$ intersect in $N$ distinct points (at least, it is
clear that they intersect in several points and the origin is isolated of multiplicity
one). Therefore
  $$
  \Tor^k_S(M_\xi,\QQ_0) = \Tor^k_S(\QQ_0,\QQ_0) \, .
  $$
This is easily computed to have dimension $\binom{r}{k}$ (using the Koszul complex).
So, using also lemma \ref{lem:0andr},
  \begin{align*}
  \dim \Tor^*_S(M,\QQ_0) &\geq (n+1) + \sum_{k=1}^{r-1}  \dim \Tor^k_S(M,\QQ_0) + (n+1) \\
   & \geq 2n+ 2 + \sum_{k=1}^{r-1}  \dim \Tor^k_S(M_\xi,\QQ_0) \\
&= 2n+2 + \sum_{k=1}^{r-1}  \binom{r}{k} = 2n+ 2^r \geq 2n+r \, .
   \end{align*}

\begin{remark} \label{rem:rem}
   The above computation works for $r\geq 1$. If $r=0$ then we have to prove that $\length(M)\geq 2n$.
 But then computing the degree $2$ non-zero elements in $M$, we have that they are at least $\binom{n+1}{2}-n$.
So for any $n$,
$$
\length(M)\geq 1+n+\binom{n+1}{2}-n= \frac12 (n+1) n +1 \geq 2n.
$$
\end{remark}

%%%%%%%%%%%%%%%%%%%%%%%%%%%%%%%%%%%%%%%%%%%%%%%%%%%%%%%%%%%%%%%%%%%%
\section{The hyperelliptic case} \label{sec:proof2}
%%%%%%%%%%%%%%%%%%%%%%%%%%%%%%%%%%%%%%%%%%%%%%%%%%%%%%%%%%%%%%%%%%%%

A minimal model $(\Lambda V,d)$ of elliptic type is \emph{hyperelliptic} if $V=V^{even}\oplus V^{odd}$,
and
 \begin{equation}\label{eqn:hyperelliptic}
  d(V^{even})=0 , \quad d(V^{odd}) \subset \Lambda^+ V^{even}\otimes \Lambda V^{odd}\, .
 \end{equation}
An elliptic space is hyperelliptic if its minimal model is so.
Note that elliptic spaces of pure type are in particular hyperelliptic.

By proposition \ref{prop:mod}
we have that $\dim V^{even}-\dim V^{odd} \leq 0$. Let $n=\dim V^{even}$ and
$n+r=\dim V^{odd}$, where $r\geq 0$.
Write $x_1,\ldots, x_n$ for
the generators of even degree, and $y_1,\ldots, y_{n+r}$ for the generators of
odd degree. Then $dx_i=0$, and $dy_j=P_j(x_1,\ldots, x_n, y_1,\ldots, y_{j-1})$, where $P_j$
do not have linear terms.

In this section we prove the following:

\begin{theorem}\label{thm:sec-proof2}
 The Hilali conjecture holds for hyperelliptic spaces.
\end{theorem}

\subsection{Semicontinuity for elliptic minimal models}

\begin{lemma} \label{lem:semicontinuity2}
 Let $V$ be a graded rational finite-dimensional vector space, and let $d$ be a differential for
  $\Lambda V \ox \QQ[t]$ such that $dt=0$, where $t$ has degree $0$.
  Take a non-numerable field $\bk \supset \QQ$, $V_\bk=V\otimes \bk$.
  We denote by $d_\xi$ the differential
  induced in $\Lambda V_\bk = \Lambda V\ox \bk[t]/(t-\xi)$, for $\xi\in \bk$. Then
    $$
     \dim H(\Lambda V_\bk, d_\xi)\leq \dim H(\Lambda V, d_0)\, ,
    $$
 for generic $\xi\in \bk$.
\end{lemma}

\begin{proof}
  Write
   $$
   0 \to \tilde K \to \Lambda V \ox \bk [t] \to \tilde I \to 0\, ,
   $$
  where $\tilde K$ and $\tilde I$ are the kernel and image of $d$, resp.
  Note that both $\tilde K$ and $\tilde I$ are free $\bk[t]$-modules, being submodules of
  $\Lambda V \ox \bk[t]$.

  Denote by $\bk_\xi= \bk[t]/(t-\xi)$. Then we have a diagram
   \begin{equation} \label{eqn:diagram}
   \begin{array}{ccccccccc}
    0 & \to & \tilde K \ox \bk_\xi & \to &
    (\Lambda V \ox \bk[t]) \ox \bk_\xi & \to & \tilde I \ox \bk_\xi & \to &0\\
     && \downarrow && ||  && \downarrow \\
     0 & \to & K  & \to & \Lambda V_\bk & \to & I  & \to &0.
     \end{array}
     \end{equation}
   (Here the tensor products of all $\bk[t]$-modules are over $\bk[t]$, and the tensor product
   $\Lambda V \ox \bk[t]$ is over the rationals.)
    Therefore the last vertical map is a surjection, and the first map is an injection.

    We have
     $$
     0 \to \tilde I \to \tilde K  \to H(\Lambda V \ox \bk[t],d) \to 0\, ,
     $$
    which is an exact sequence of $\bk[t]$-modules. Then $H(\Lambda V \ox \bk[t],d)$
    contains a free part and a torsion part. The torsion is supported at some points,
    which are at most countably many. Therefore for generic $\xi\in \bk$,
       $$
     0 \to \tilde I\ox \bk_\xi \to \tilde K\ox \bk_\xi  \to H(\Lambda V \ox \bk[t],d)\ox \bk_\xi \to 0
     $$
     is exact.
   As $\tilde I\ox\bk_\xi \surj I \subset K$ and $ \tilde I\ox \bk_\xi \subset \tilde K\ox \bk_\xi
   \subset K$, we have that the last map in (\ref{eqn:diagram}) is an injection, therefore an isomorphism,
 thus first map is also an isomorphism by the snake lemma.

  Note that also, when tensoring with $\bk(t)$, we have an exact sequence
          $$
     0 \to \tilde I\ox \bk(t) \to \tilde K\ox \bk(t)  \to H(\Lambda V \ox \bk[t],d)\ox \bk(t) \to 0\, .
     $$
     Also $H(\Lambda V \ox \bk[t],d)\ox \bk(t)=H(\Lambda V \ox \bk(t),d)$, since $\bk(t)$ is a flat
   $\bk[t]$-module. Hence
     \begin{align*}
     \dim H(\Lambda V_\bk, d_\xi) & =\dim K-\dim I \\
     &= \dim \tilde K\ox \bk_\xi -\dim \tilde I\ox \bk_\xi \\
     &= \dim H(\Lambda V \ox \bk(t),d)\, .
     \end{align*}
     In the first line, we mean $\dim K-\dim I= \sum_{d\geq 0} (\dim K^d - \dim I^d)$.

     Take now $\xi =0$.
     The map $\tilde K \to K \to K/I$ factors as $\tilde K/\tilde I \to K/I$. Tensor this map by
     $\bk_0$ to get $(\tilde K/\tilde I) \ox \bk_0 \to K/I$. Note that there is an exact sequence
     $$
     \tilde I \ox \bk_0 \to \tilde K \ox \bk_0 \to (\tilde K/\tilde I) \ox \bk_0 \to 0,
     $$
     but the first map may not be injective. Then there is a map
      $$
     \frac{\tilde K \ox \bk_0}{\mathrm{Im} (\tilde I\ox\bk_0)} = (\tilde K/\tilde I) \ox \bk_0  \to K/I \, .
     $$
     By (\ref{eqn:diagram}), this is an inclusion. Now we have:
     \begin{align*}
     \dim H(\Lambda V, d_\xi) &= \dim H(\Lambda V\ox \bk(t),d) \\
      &= \dim (\tilde K/{\tilde I}) \ox \bk(t)  \\
      & \leq  \dim ({\tilde K}/{\tilde I}) \ox \bk_0 \\
      & =  \dim \frac{\tilde K \ox \bk_0}{\mathrm{Im} (\tilde I\ox\bk_0)} \\
      & \leq \dim K/I \\
      &= \dim H(\Lambda V_\bk, d_0) \\
      &= \dim_\QQ H(\Lambda V , d_0) \, .
    \end{align*}
\end{proof}

\subsection{Perturbing the minimal model}
 Let $x_1,\ldots, x_n$ denote generators for $V^{even}$, and $y_1,\ldots, y_{n+r}$ generators for
 $V^{odd}$. Here $dx_i=0$ and $dy_j=P_j(x_1,\ldots, x_n, y_1, \ldots, y_{j-1})$.

 We consider the algebra
 $$
  (\Lambda W,d)= (\Lambda V, d)\ox (\Lambda \bar{y}_1,0)\, ,
  $$
   where $\deg(\bar{y}_1)=\deg(x_1)-1$. Then
     $$
  \dim H(\Lambda W,d)= 2 \dim H(\Lambda V, d)\, .
  $$

 Consider now the differential $\delta$ on $\Lambda W$ such that $\delta x_j=0$, $\delta y_j=0$
 and $\delta \bar{y}_{1}=x_1$. Hence $\delta^2=0$ and $d\delta=\delta d= 0$. So
   $$
   d_t=d + t\delta
   $$
 is a differential on $\Lambda W \ox \bk[t]$.

 For generic $\xi\in \bk$, $(\Lambda W_\bk, d_\xi)$ verifies
 that $d_\xi \bar{y}_{1}= \xi x_1$. So for non-zero $\xi$,
 there is a KS-extension \cite[\S1.4]{OT}
  $$
     (\Lambda (x_1,\bar{y}_{1}),d_\xi ) \too (\Lambda W_\bk, d_\xi)
    \too (\Lambda (x_2,\ldots, x_n, y_1,\ldots y_{n+r}), d) \, .
     $$
     As $H( \Lambda (x_1,\bar{y}_{1}),d_\xi )=\bk$, we have that
  \begin{equation*}%\label{eqn:reduction}
  H(\Lambda W_\bk, d_\xi) \cong H(\Lambda (x_2,\ldots, x_n, y_1,\ldots y_{n+r}), d) \, .
  \end{equation*}
% This is proved as follows: define the element $x_1'=x_1 + \frac1\xi d y_{n+r}$. Then
%  $\Lambda V \cong \Lambda (x_1',x_2,\ldots, x_n,y_1,\ldots y_{n+r}$, and $d_\xi y_{n+r}=\xi x_1'$.
%  (Note that the differential $d_\xi$ still satisfies (\ref{eqn:hyperelliptic}) with respect to
%  the new decomposition).
%   Now we have the KS-extension
%   As  the claim (\ref{eqn:reduction}) follows.

  Now we apply lemma \ref{lem:semicontinuity2} to this to obtain that
   $$
   \dim H(\Lambda (x_2,\ldots, x_n,y_1,\ldots y_{n+r}), d) \leq \dim H(\Lambda W,d)=2\dim H(\Lambda V,d)\, .
   $$
  Repeating the argument $n$ times, we get that
     $$
   \dim H(\Lambda (y_1,\ldots y_{n+r}), d) \leq 2^n \dim H(\Lambda V,d)\, .
   $$
  But the hyperelliptic condition says that $d=0$ for the first space, so
   $$
   2^n\dim H(\Lambda V,d)\geq \dim H(\Lambda (y_1,\ldots y_{n+r}), d) =2^{n+r} \, .
   $$
   This gives
   \begin{equation}\label{eqn:2r}
   \dim H(\Lambda V,d)\geq 2^r \, .
   \end{equation}

\subsection{Another proof of (\ref{eqn:2r})}

In this paragraph we present a different proof of the inequality $\dim H(\Lambda V,d)\geq 2^r$
for hyperelliptic spaces. Recall that if $A$ is a commutative graded differential algebra, and
if $M,N$ are differential graded $A$-modules, the differential Tor is defined as:
$$
\Tor^\ast(M,N)=H^\ast(P\otimes_A N),
$$ where $P\stackrel{\sim}{\longrightarrow} M$ is a semifree resolution, i.e.,
a quasi-isomorphism from a semifree $A$-module $P$ to $M$ (see \cite[\textsection 6]{FHT}).

\begin{lemma} Let $C\stackrel{\varphi}{\longleftarrow}A\stackrel{\psi}{\longrightarrow}B$ be morphisms
of commutative differential graded algebras. There exists a convergent spectral sequence:
$$
E_2^{p,q}=H^p(B)\otimes \Tor^q_A(\QQ,C)\Rightarrow \Tor^{p+q}_A(B,C).
$$
\end{lemma}

\begin{proof}
Decompose $\varphi$ and $\psi$ as:
$$
\xymatrix
  {A\ar[dr]_{\psi}\ar[r] & A\otimes\Lambda W\ar[d]^{\sim}_{\alpha}&&A \ar[dr]_{\varphi}\ar[r] &A\otimes\Lambda U\ar[d]^{\sim}_{\beta}\\
   & B && &C}
  $$
Then $\alpha\colon A\otimes \Lambda W\stackrel{\sim}{\longrightarrow}B$ is a semifree resolution of $B$ regarded as $A$-module, so
$$
\Tor^*_A(B,C)=H^*((A\otimes\Lambda W)\otimes_AC).
$$

Moreover, $\text{Id}\otimes \beta\colon \
(A\otimes\Lambda W)\otimes_AA\otimes\Lambda U\stackrel{\sim}{\longrightarrow} (A\otimes\Lambda W)\otimes_AC$
is a quasi-isomorphism and
$(A\otimes\Lambda W)\otimes_A(A\otimes\Lambda U)\cong A\otimes\Lambda W\otimes\Lambda U.$ Therefore one gets a rational fibration
$$
A\otimes\Lambda W\to A\otimes\Lambda W\otimes\Lambda U\to \Lambda U,
$$
whose associated Serre spectral sequence has the form
$$
E_2^{p,q}=H^p(A\otimes\Lambda W)\otimes H^q(\Lambda U)\Rightarrow H^{p+q}(A\otimes\Lambda W\otimes\Lambda U).
$$

On the one hand, $H^*(A\otimes\Lambda W)=H^*(B).$ On the other hand, since $\beta$ is a semifree resolution of $C,$ we have that:
$$
H^*(\Lambda U)=H^*((A\otimes\Lambda U)\otimes_A\QQ)=\Tor^*_A(\QQ,C).
$$

Putting all pieces together we get
$$
E_2^{p,q}=H^p(B)\otimes \Tor^q_A(\QQ,C)\Rightarrow \Tor^{p+q}_A(B,C)
$$
\end{proof}

\begin{theorem} Let $(\Lambda V, d)$ be a hyperelliptic minimal model. Then
$$
\dim H(\Lambda V, d)\geq 2^r.
$$
\end{theorem}

\begin{proof}
Write as usual $x_1,\ldots,x_n$ for generators of $X=V^{even}$ and $y_1,\ldots,y_{n+r}$ for generators of $Y=V^{odd}$. When we apply the previous lemma to morphisms $\QQ{\longleftarrow}\Lambda X{\hookrightarrow}\Lambda V$ we get a spectral sequence:
$$
E_2=H(\Lambda V, d)\otimes \Tor^*_{\Lambda X}(\QQ,\QQ)\Rightarrow \Tor^*_{\Lambda X}(\Lambda V,\QQ).
$$

On the one hand,
$$
\Tor^*_{\Lambda X}(\QQ,\QQ)=H^*(\Lambda (\overline{x}_1,...,\overline{x}_n),0)=\Lambda (\overline{x}_1,\ldots,\overline{x}_n),
$$
where $\Lambda (x_1,...,x_n,\overline{x}_1,...,\overline{x}_n)\stackrel{\sim}{\longrightarrow} \QQ$ is a semifree resolution of $\QQ$ regarded as $\Lambda X$-module. Hence $\overline{x}_i$ are all of odd degree.

On the other hand, $\Lambda V$ is already $\Lambda X$-semifree, so:
$$
\Tor^*_{\Lambda X}(\Lambda V,\QQ)=H(\Lambda V\otimes_{\Lambda X}\QQ)=H^*(\Lambda (y_1,...,y_{n+k}),0)=\Lambda (y_1,...,y_{n+k}).
$$

Then the inequality
$$\dim H^*(\Lambda V, d)\cdot\dim \Tor^*_{\Lambda X}(\QQ,\QQ) \geq \dim \Tor^*_{\Lambda X}(\Lambda V,\QQ)$$
coming from the spectral sequence translates into
$$2^n\dim H^*(\Lambda V, d)\geq 2^{n+r},$$
so the result follows.
\end{proof}

\subsection{Proof of theorem \ref{thm:sec-proof2}}
   Now we prove the inequality $\dim H(\Lambda V,d)\geq 2n+r$, for the hyperelliptic
   minimal model.

   If $r=0$, then $\chi_\pi=0$. So \cite[Prop. 32.10]{FHT} says that the model is pure,
   and this case is already covered by remark \ref{rem:rem}.

   If $r>0$, then $\chi_\pi<0$. So by proposition \ref{prop:mod}, $\chi=0$, and hence it is enough
   to prove that
    $$
    \dim H^{even}(\Lambda V,d) \geq n+\tfrac{r}{2}.
    $$

   Suppose that  $r=1,2$. As the degree $0$ and degree $1$ elements give always non-trivial
   homology classes,
   then $\dim H^{even}(\Lambda V,d) \geq n+1$, and we are done.

   So we can assume $r\geq 3$. We use the following fact: if $P(x)$ is a quadratic
   polynomial on the $x$, and  $P(x) =d \alpha$, $\alpha \in \Lambda V$, then $\alpha$ must
   be linear, $\alpha\in V^{odd}$ and denoting by $d_o$ the composition
    $$
    V^{odd}\too  \Lambda^+ V^{even}\otimes \Lambda V^{odd} \surj \Lambda^+ V^{even}\, ,
    $$
   we have $P(x)=d_o \alpha$. So there are at least $\binom{n+1}{2} - (n+r)$ quadratic
   terms in the homology. Conjecture \ref{conj:Hilali} is proved if
    \begin{equation} \label{eqn:30}
    \left\{ \begin{array}{rl} \text{either} & 1+ n + \binom{n+1}{2} - (n+r) \geq n+\frac{r}2 \, , \\
    \text{or} & 2^r \geq 2n + r \, . \end{array}\right.
    \end{equation}

    So now assume that (\ref{eqn:30}) does not hold. Then
      \begin{equation}\label{eqn:3a}
      2^r -r \leq 2n-1\, ,
     \end{equation}
    and $1+ \binom{n+1}{2} - n <\frac32 r$, i.e.,
     \begin{equation}\label{eqn:3b}
     (2n-1)^2 \leq 12r-11.
     \end{equation}
  Putting together (\ref{eqn:3a}) and (\ref{eqn:3b}), we get $2^r-r \leq \sqrt{12r-11}$, i.e.,
   $2^r \leq r+ \sqrt{12r-11}$. This is easily seen to imply that $r\leq 3$. So $r=3$ and $n=3$.

   There remains to deal with the case $n=3$, $r=3$, and $d_o$ is an isomorphism of the
   odd degree elements onto $\Lambda^2 V^{even}$. Let $x_1,x_2,x_3$ be the even degree
   generators, of degrees $d_1\leq d_2\leq d_3$ respectively. The degrees of $x_1^2, x_1x_2, x_2^2 ,
   x_1x_3, x_2x_3,x_3^2$ are the six numbers
    $$
    2d_1\leq d_1+d_2 \leq 2d_2, \quad d_1+d_3 \leq d_2+d_3 \leq 2d_3.
    $$

We have two cases:
\begin{itemize}
 \item Case $2d_2\leq d_1+d_3$.
   We can arrange the odd generators $y_1,\ldots, y_6$ with  increasing degree and so that
   $d_oy_1=x_1^2, d_oy_2=x_1x_2, d_oy_3=x_2^2 ,
   d_oy_4=x_1x_3, d_oy_5=x_2x_3,d_oy_6=x_3^2$. Clearly, $d y_1=x_1^2$. Then $dy_2=x_1x_2 + P(x_1)$,
   where $P(x_1)$ is a polynomial on $x_1$, i.e., of the form $cx_1^n$, $n\geq 2$. But this can
   absorbed by a change of variables $y_2\mapsto y_2-cx_1^{n-2}y_1$. So we can write $dy_2=x_1x_2$.
   Now the even-degree closed elements in $\Lambda (x_1,x_2,x_3,y_1,y_2)$ are again polynomials on $x_1,x_2,x_3$.
   So we can assume $dy_3=x_2^2$ as before. Continuing the computation,
   the even-degree closed elements in $\Lambda (x_1,x_2,x_3,y_1,y_2,y_3)$ are either polynomials on the
    $x_i$'s or a multiple of the element
   $x_2^2y_1y_2- x_1x_2 y_1y_3 + x_1^2 y_2y_3= d(y_1y_2y_3)$, which is exact. So we can again manage to
   arrange that $dy_4=x_1x_3$.
\item Case $2d_2> d_1+d_3$. Then we have that $d_oy_3=x_1x_3$ and
   $d_oy_4=x_2^2$. As before, we can arrange $dy_3=x_1x_3$. Now the even-degree closed elements in
   $\Lambda (x_1,x_2,x_3,y_1,y_2,y_3)$
   are polynomials on the $x_i$'s or a multiple of
   $x_3y_1y_2 -x_2y_1y_3 + x_1y_2y_3$. But this element has degree $3d_1+d_2+d_3-2 > 2d_2$,
   so it must be $dy_4=x_2^2$.
\end{itemize}

  In either case, $dy_1,dy_2,dy_3,dy_4$ are $x_1^2,  x_1x_2, x_2^2$ and
   $x_1x_3$. Let us assume that we are in the first case to carry over the notation.

   Now we compute the even-degree closed elements in $\Lambda (x_1,x_2,x_3,y_1,y_2,y_3, y_4)$. These
are polinomials on $x_i$'s or combinations of
  \begin{align*}
    & x_2^2y_1y_2- x_1x_2 y_1y_3 + x_1^2 y_2y_3= d(y_1y_2y_3), \\
    & x_3y_1y_2 -x_2y_1y_4 + x_1y_2y_4, \\
   &x_1x_3y_2y_3-x_2^2y_2y_4+x_1x_2y_3y_4=d(y_2y_3y_4), \text{ and} \\
   & x_1x_3y_1y_3 +x_1^2y_3y_4 -x_2^2y_1y_4=d(y_1y_3y_4).
  \end{align*}
  Only the second one is non-exact, but its degree is strictly bigger thatn $d_2+d_3$.
  So again we can arrange that $dy_5=x_2x_3$.

   Finally, the minimal model is:
   $$
   \left\{ \begin{array}{l} dy_1=x_1^2, \\ dy_2=x_1x_2, \\ dy_3=x_2^2 , \\
   dy_4=x_1x_3,\\ dy_5=x_2x_3, \\ dy_6=x_3^2 +  P(x_i,y_j).
\end{array} \right.
   $$
   The even-degree closed elements in $\Lambda (x_1,x_2,x_3,y_1,y_2,y_3, y_4, y_5)$ contain at least
   \begin{eqnarray*}
     \alpha_1 &=&  x_3y_2y_3+x_1y_3y_5-x_2y_2y_5 \, ,\\
     \alpha_2 &=&  x_3y_1y_2 -x_2y_1y_4+ x_1y_2y_4\, .
   \end{eqnarray*}
   At most one of them does not survive in $H(\Lambda V,d)$,
so proving the existence of at least another even-degree cohomology class. Hence $\dim H(\Lambda V,d)\geq 10\geq 9$,
as required.

\end{document}